\documentclass[11pt]{amsart}
\usepackage{amssymb,amsmath}

\usepackage{hyperref}
\usepackage{epsf,pstricks,multicol,pst-grad,color,pst-node,pst-text,pst-3d,pst-tree}

\numberwithin{equation}{section}
\numberwithin{figure}{section}

\usepackage[foot]{amsaddr}

\newcommand{\cW}{\mathcal{W}}
\newcommand{\cD}{\mathcal{D}}

\newcommand{\eps}{\varepsilon}

\newcommand{\CO}{{\mathcal O}}

\newcommand{\N}{{\mathbb N}}

\newtheorem{theorem}{Theorem}[section]
\newtheorem{lemma}[theorem]{Lemma}
\newtheorem{definition}[theorem]{Definition}
\newtheorem{assumption}[theorem]{Assumption}
\newtheorem{prop}[theorem]{Proposition}
\newtheorem{remark}[theorem]{Remark}
\newtheorem{conj}[theorem]{Conjecture}

\newcommand{\db}[1]{{#1}}
\newcommand{\da}[1]{{#1}}
\newcommand{\dbz}[1]{{#1}}

\begin{document}
\title[Sharp interface limit for Stochastic Cahn-Hilliard equation]{
The Sharp interface limit for the\\ Stochastic Cahn-Hilliard
equation}

\author[Antonopoulou]{D.C.~Antonopoulou$^{\$*}$}
\address{$^{\$}$ Department of Mathematics, University of Chester,
Thornton Science Park, CH2 4NU, UK}
\author[Bl\"omker]{D.~Bl\"omker$^{\ddag}$}
\address{$^{\ddag}$Institut f\"ur Mathematik Universit\"at Augsburg, D-86135 Augsburg.}
\author[Karali]{G.D.~Karali$^{\dag *}$}
\address{$^{\dag}$ Department of Mathematics and Applied Mathematics,
University of Crete, Heraklion, Greece.}
\address{$^{*}$ Institute of Applied and Computational Mathematics,
FO.R.T.H., GR--711 10 Heraklion, Greece.}

\email{d.antonopoulou@chester.ac.uk, dirk.bloemker@math.uni-augsburg.de, gkarali@uoc.gr}


\date{\today}


\begin{abstract}
We study the $\varepsilon$-dependent two and three dimensional
stochastic Cahn-Hilliard equation in the sharp interface limit
$\varepsilon\rightarrow 0$. The parameter $\varepsilon$ is
positive and measures the width of transition layers generated
during phase separation. We also couple the noise strength to
this parameter. Using formal asymptotic expansions, we identify
the limit. In the right scaling, our results indicate that the
stochastic Cahn-Hilliard equation converge to a Hele-Shaw problem
with stochastic forcing on the curvature equation. In the case
when the noise is sufficiently small, we rigorously prove that the
limit is a deterministic Hele-Shaw problem. Finally, we discuss
which estimates are necessary in order to extend the rigorous
result to larger noise strength. \bigskip

\noindent{\sc R\'esum\'e:} Nous \'etudions l' \'equation Cahn-Hilliard
stochastique d\'ependante en $\eps$, pos\'ee en dimensions deux et
trois dans la limite de l'interface nette $\eps\rightarrow 0$. Le
param\`etre $\eps$ est positif et mesure la largeur de couches de
transition g\'en\'er\'ees pendant la s\'eparation de phase. Nous
avons aussi coupl\'e la puissance de bruit \`a ce param\`etre.
\`A l'aide de s\'eries asymptotiques formelles, nous
d\'eterminons la limite. Dans l'\'echelle propre, nos r\'esultats
indiquent que l' \'equation Cahn-Hilliard stochastique converge
vers un probl\`eme Hele-Shaw avec une force stochastique
pr\'esente \`a l'\'equation de la courbure. En cas de bruit
suffisamment petit, nous prouvons rigoureusement que la limite
est un probl\`eme Hele-Shaw d\'eterministe. Finalement, nous
discutons quelles estimations sont n\'ecessaires afin de
prolonger le r\'esultat rigoureux en pr\'esence de bruit d'une
puissance plus grande.
\end{abstract}

\maketitle

%
\section{Introduction}
%
%
In this paper we consider the sharp interface limit of the
stochastic Cahn-Hilliard equation
\begin{equation}
\label{main}
\partial_t u=\Delta
(-\eps \Delta
 u+\eps^{-1}f^{\prime}(u)) +\db{ \eps^\sigma\dot{\mathcal{W}}(x,t),}
\end{equation}
\dbz{for times $t\in[0,T]$}
subject to Neumann boundary conditions on a bounded domain in $\mathcal{D}\subset\mathbb{R}^d$, $d\in\{2,3\}$
\begin{equation}\label{ncond}
\displaystyle\frac{\partial u}{\partial
n}=\displaystyle\frac{\partial \Delta u}{\partial n}=0 \,\,\,\,\,
\mbox{on} \,\,\,\, \partial\mathcal{D}\;.
\end{equation}
\dbz{Here $u: \mathcal{D}\times[0,T]\to\mathbb{R}$ is the scalar concentration field of one of the components in a separation process,
for example of binary alloys.}

We always assume that \dbz{the spatial domain} $\mathcal{D}$  has a sufficiently piece-wise smooth boundary.
The typical example for the bistable nonlinearity is
$f^{\prime}(u):=\partial_u f(u)=u(u^2-1)$, where the primitive
$f:=\frac{1}{4}(u^2-1)^2$ is a double well potential with equally deep wells taking
its global minimum value $0$ at the points $u=\pm 1$.
The small parameter $\eps$ measures an atomistic interaction length
that fixe\db{s} the length-scale of transition layers between $1$ and $-1$.
Obviously\db{,} the solution $u$ depends on $\eps$, which we usually suppress in the notation,
but wherever needed we denote the solution $u$ by $u_{\eps}$.

\db{The forcing $\dot{\mathcal{W}}$ denotes \da{a} Gaussian noise,
which is white in time and possibly colored in space. Finally,
the noise-strength $\eps^\sigma$ scales with the atomistic length
scale, where $\sigma$ is some parameter.}

\db{In applications, there are often different sources of the
noise. If we split \dbz{the noise} $\eps^\sigma \dot{\mathcal{W}}=
\Delta\dot{\mathcal{V}}+\dot{W}$, then the term
$\dot{\mathcal{V}}$ arises from fluctuations in the chemical
potential}\dbz{, while $\dot{W}$ models fluctuations directly in
the concentration.  Here,} \da{$\dot{\mathcal{V}}$} needs to be a
spatially smooth space-time noise since it is under the
Laplacian, while the additive noise $\dot W$  might be rougher.

The stochastic Cahn-Hilliard equation is a model for the
non-equilibrium dynamics of metastable states in phase
transitions, \cite{cook,hohenberg,langer}.
The deterministic Cahn-Hilliard equation with \db{$\dot{\mathcal{W}}=0$}
has been extended to a stochastic version by Cook, \cite{cook}
(see also in \cite{langer}), incorporating thermal fluctuations in
the form of an additive noise. Such a generalized Cahn-Hilliard
model, \cite{Gurtin}, is based on the balance law for micro-forces.
In this case, the additive term \db{$\dot{\mathcal{V}}$} in the chemical potential
is given by fluctuations in an external field.
See \cite{hohenberg, Gurtin}; cf. also  \cite{Kitahara},  where the external gravity field is
modeled. The \dbz{noise $\dot W$, which is independent of the} free energy, stands for the
Gaussian noise in Model B of \cite{hohenberg}, in accordance with
the original Cahn-Hilliard-Cook model.

We can model \dbz{the  noise} either as the formal derivative of
a Wiener process in the sense of Da Prato \& Zabczyck
\cite{DPZa:14} or as a derivative of a Brownian sheet in the
sense of Walsh, \cite{Wal86}. Dalang \& Quer-Sardanyons
\cite{DaQS11} showed that both approaches are actually
equivalent\da{. Thus, we} focus on the $Q$-Wiener-process \da{in}
the sense of \cite{DPZa:14}, which may be  defined by a series in
the right orthonormal basis with independent coefficients given
by a sequence of real valued Brownian motions.

Existence of stochastic solution for the problem \eqref{main} has
been established under various assumptions on the
noise-terms\da{, see} for example in  \cite{DPDe:96,El:91}; we
also refer to the results presented in
\cite{Weber22,Weber12,ak,DBMPWa:01}.

%
\subsection{Phase transitions and noise}
%
%
Concerning the Cahn-Hilliard equation  posed in one
dimension, in \cite{abk} the authors analyzed the stochastic
dynamics of the front motion \dbz{of one dimensional interfaces.} In the absence of noise we refer also to
\cite{BaXu:95,BaXu94} for the dynamics \dbz{of interfaces}  and \dbz{the} construction of a
finite dimensional manifold parametrized by the interface
positions. This is a key tool for studying the stochastic case,
but it fails in dimension two and three, as interfaces are no
longer points, but \da{curves or surfaces respectively}.

An interesting result is \cite{BeBrBu:14}, where,
on the unbounded domain, a single interface moves according to a
fractional Brownian motion, which is in contrast to the usual Brownian motion in most of the \dbz{other} examples.
Note that the one-dimensional case is
significantly simpler, since the solutions can be fully
parametrized by their finitely many zeros. Therefore, one needs
only to consider the motion of solutions along a finite
dimensional slow manifold, and the stability properties along
such a manifold; see for example \cite{BaLuZe:08}.
Here we try to follow similar ideas, despite the fact that the driving manifold is infinite dimensional\dbz{,
and parametrized by the curves or surfaces of the zero level set.}

A slightly simpler model, due to the absence of mass-conservation,
is the stochastic Allen Cahn-equation, the so called Model A of
\cite{hohenberg}. The interface motion of stochastic systems of
Allen-Cahn type have been analyzed in \cite{faris}. In
\cite{Fun95,bmp}, the authors studied the stochastic
\dbz{one-dimensional} Allen-Cahn equation with initial data close
to one instanton or interface and proved that, under an
appropriate scaling, the solution will stay close to the
instanton shape, while the random perturbation will create a
dynamic motion for this single interface. This is observed on a
much faster time scale than in the deterministic case. This
result has been also studied in \cite{weber2} via an invariant
measure approach.

If the initial data involves more than one interfaces, it is
believed that these interfaces exhibit \dbz{also} a random movement,
which is much quicker than in the deterministic case, while
different interfaces should annihilate when they meet,
\cite{fatkullin}. We also refer to \cite{Sha:00} or
\cite{HaLy:00}. The limiting process should be related to a
Brownian one (cf. in \cite{fontes} for formal arguments). A full description of all the ideas for the
analysis of the interface motion based on \cite{chen} and
\cite{abk} is presented in \cite{weber:phd}. In \cite{rogweb} the authors considered
the stochastic Allen-Cahn equation driven by a
multiplicative noise; they prove tightness of solutions
for the sharp interface limit problem, and show convergence to
phase-indicator functions; cf. also in \cite{weber2} for the
one-dimensional case with an additive space-time white noise for
the proof of an exponential convergence towards a curve of
minimizers of the energy.

The space-time white noise driven Allen-Cahn equation is known to
be ill-posed in space dimensions greater than one, \cite{Wal86,
DPZa:14}, and a renormalization is necessary to properly define
the solutions. We refer to \cite{Ha:14,HaWe:14} for more details.

For a multi-dimensional stochastic Allen-Cahn equation driven by
a mollified noise, in \cite{hiweb}, it is shown that as the
mollifier  is removed, the solutions converge weakly to
zero, independently of the initial condition. If the the noise
strength converges to zero at a sufficiently fast rate, then the
solutions converge to those of the deterministic equation,
and the behavior is well described by the  Freidlin-Wentzell theory.
\dbz{See also \cite{BeWe:P}.}
In
\cite{weber1}, for a \dbz{regularized noise, which is white in the limit,} the author
extending the classical result of Funaki \cite{Fun99} to spatial
dimensions more than two, derived motion by mean curvature with an
additional stochastic forcing for the sharp interface limit
problem. Recently, in \cite{AAKM}, for the case of an additive
`mild' noise in the sense of \cite{Fun99,weber1}, the first
rigorous result on the generation of interface for the stochastic
Allen-Cahn equation has been derived; the authors proved layer
formation for general initial data, and established that the
solution's profile near the interface remains close to that of a
(squeezed) traveling wave, which means that a spatially uniform
noise does not destroy this profile.
%
%
\subsection{\db{Main Results}}
%
%
Due to phase separation, the solution $u$ of the stochastic
Cahn-Hilliard equation \eqref{main}, which is related to the mass
concentration, tends to split in regions where $u\approx\pm1$ and
with inner layers of order $\mathcal{O}(\eps)$ between them. We
shall study the motion of such layers in their sharp interface
limit $\eps\rightarrow 0$. \textit{The rigorous complete
description of the motion of interfaces in dimensions two and
three stands for many years as a wide open question}. With this
paper, we tried to contribute towards a full answer by providing
\db{first some formal} asymptotics for \db{the case of general
noise strength}. \db{We \da{were} able to identify the
intermediate regime $\sigma=1$ when in the limit
$\eps\to0$\da{,} the limiting solutions satisfy a
stochastic
 Hele-Shaw problem, \da{ (see \eqref{H-S-nh})}.
While for smaller noise strength $\sigma>1$ we conjecture that the limit solves a deterministic Hele-Shaw problem.}

In addition, as a first \db{rigorous} step, we derive
the sharp interface limit \dbz{given by the deterministic Hele-Shaw problem in the case} when the noise is sufficiently small,
\da{ (see Theorem \ref{thm:main})}. Unfortunately, we
are not able to treat the case $\sigma$ close to $1$,
\da{ we refer to Remark \ref{rem1}}.  A key problem is
that spectral problems of the linearized operators are not yet
understood in the \da{general setting which is} necessary for
the proof, (\db{see Remark \ref{rem2}).}

\db{The result for \da{the} Cahn-Hilliard equation in the
deterministic case (when $\dot{\mathcal{W}}=0$)} has been
already studied in \cite{abc}. Given a solution of the deterministic  Hele-Shaw
problem, the authors constructed an approximation of (\ref{main}) without noise
admitting this solution, as the
interface moves between $\pm1$. The analysis thereof will be the
foundation of some of our results. The main technical problem of
this strategy, is that the manifold of
possible approximations is parametrized by an infinite
dimensional space of closed curves \dbz{or surfaces}.
Furthermore, the spectral
information provided so far for the linearized problem, necessary
for a qualitative study of the approximation is insufficient.
This is due mainly to the fact that most of the larger
eigenvalues actually are related to the fast motion of the
interface itself\dbz{, which do not obstruct the stability of the manifold}.
Thus, this approximation can only be valid  on
time-scales of order $\mathcal{O}(1)$.
\db{\da{ We refer to} the main result in Theorem \ref{thm:main}, where we
approximate solutions of the stochastic
Cahn-Hilliard equation by solutions of the deterministic Hele-Shaw problem
only for very small noise and only on $\mathcal{O}(1)$-time-scales.}

A simpler case is the motion of droplets for the two-dimensional Cahn-Hilliard or
the mass conserving Allen-Cahn equation. Here, the solutions can
be fully parametrized by finite dimensional data, namely the
position and radius of the droplets. See \cite{AlChFu:00,
BaJi:12}, and \cite{AnBaBlKa:15} (stochastic problem), for
droplets on the boundary, and \cite{AlFuKa:04, AlFu:98} for
droplets in the domain in the absence of noise. The approximation
in these cases is valid for very long time-scales.


\subsection{Outline of the paper}


In Section \ref{sec:formal}, we present the formal
derivation of the stochastic Hele-Shaw problem from (\ref{main})
and identify the \dbz{order of the} noise strength that \dbz{leads} to a nontrivial \dbz{stochastic} modification
of the \dbz{deterministic} limiting problem.

In Section \ref{sec:main}, we provide a rigorous definition of the
setting and state the main results, which we then prove in
Section \ref{sec:proof}. We concentrate on small noise strength
and show in that case, that solutions of (\ref{main})  in
the sharp interface limit of $\eps\rightarrow 0$  are well approximated by a
Hele-Shaw problem. We will see that the main limitation towards a
better approximation result is the lack of good bounds on the
linearized operator.


\section{Formal asymptotics}
\label{sec:formal}

In this section we  present some formal matching asymptotics
applied \dbz{to} \eqref{main} that establish a first intuition
towards a rigorous proof for the stochastic sharp interface limit.
We remark first that
the stochastic C-H equation  \eqref{main} can be written as a
system\db{, where $v$ is the chemical potential.}
\begin{equation}\label{system}
\left\{
\begin{array}{lll}
\partial_t u & = & -\Delta v+ \eps^\sigma \dot{\mathcal{W}},\\
\\
v & = & -\displaystyle\frac{f^{\prime}(u)}{\eps}+\eps \Delta u.
\end{array}
\right.
\end{equation}

Formally in \cite{akk}, and later more rigorously in \cite{ako}
using a Hilbert expansions method, the asymptotic behavior for
$\eps\rightarrow 0$ of the following deterministic system has
been analyzed:
\begin{equation*}\label{system2}
\left\{
\begin{array}{lll}
\partial_t u\dbz{_\varepsilon} & = & -\Delta v\dbz{_\varepsilon}+G_1,\\
\\
v\dbz{_\varepsilon} & = & -\displaystyle\frac{f^{\prime}(u\dbz{_\varepsilon})}{\eps}+\eps \Delta
u\dbz{_\varepsilon}+G_2,
\end{array}
\right.
\end{equation*}
where now $G_1(x,t;\eps)$ and $G_2(x,t;\eps)$ are deterministic
forcing terms. The sharp interface limit problem in the
multidimensional case demonstrated a local influence in phase
transitions of forcing terms that stem from the chemical
potential, while free energy independent terms act on the rest of
the domain. In addition, the forcing may slow down the
equilibrium.

Given an initial smooth closed
$\textit{n}-1$ dimensional hypersurface $\Gamma_0$ in
$\mathcal{D}$ (this definition covers also the union of closed
interfaces) then the \dbz{limiting} chemical potential
\begin{equation}\label{pot}
v:=\displaystyle{\lim_{\eps\rightarrow 0^+}}(\eps\Delta
u\dbz{_\varepsilon}-\eps^{-1}f^{\prime}(u\dbz{_\varepsilon})+G_2),
\end{equation}
 satisfies the following Hele-Shaw free
boundary problem (assuming that the limits exist)
\begin{equation}\label{nhH-S*} \left\{
\begin{aligned}
&\Delta v=\displaystyle{\lim_{\eps\rightarrow 0^+}}G_1 \,\, \mbox{ in } \mathcal{D} \backslash\Gamma(t),\;\;t>0, \\
&\partial_n v=0 \,\, \mbox{ on }\partial\mathcal{D},\\
&v=\lambda H+\displaystyle{\lim_{\eps\rightarrow 0^+}}G_2\,\,  \mbox{ on } \Gamma(t),\\
&V=\frac{1}{2}(\partial_n v^+-\partial_n v^-) \,\,  \mbox{ on } \Gamma(t), \\
&\Gamma(0)=\Gamma_0,
\end{aligned}
\right.
\end{equation}
where $\Gamma(t)$ is the zero level surface of \dbz{the limiting} $u(t)$, which  is for fixed time \dbz{$t$} a
closed $\textit{n}-1$ dimensional hypersurface of mean curvature
$H=H(t)$ and of velocity $V=V(t)$ that divides the domain $\mathcal{D}$ in
two open sets $\mathcal{D}^+(t)$ and  $\mathcal{D}^-(t)$. The
constant $\lambda$ is positive, and $n$ is the unit outward
normal vector at the inner and outer boundaries.

According to the aforementioned arguments, each perturbation
$G_i$ has a different physical meaning and appears in a different
equation when C-H is presented as a system. We will use some of
the ideas of the deterministic asymptotic analysis, but we will
see in the following that \da{when} the terms $G_i$ are noise
terms and small in $\eps$, they still have an impact on the
limiting behavior.

\subsection{Formal derivation of the stochastic Hele-Shaw
problem}\label{formal}


In order to observe the limit behavior of \eqref{main} at larger
noise strengths, we fix $\sigma=1$ as we expect the noise
strength $\eps$ to be the critical one. In order \da{to avoid
calculations} with formal noise terms, where the order of
magnitude and the definition of products is not always obvious,
we use the following change of variables
$$\hat{u}_\eps:=u_\eps - \eps{\cW}\;,
$$
where we assume that the Wi\db{e}ner process is spatially sufficiently smooth.
Taking differentials, it follows that $\hat{u}_\eps$ and $v_\eps$
solve the system
\begin{equation}\label{CH:sys_sub}
\left\{
\begin{array}{lll}
\partial_t \hat{u}_\eps & = & -\Delta v_\eps, \\
\\
v_\eps & = & -\displaystyle\frac1{\eps}
f^{\prime}(\hat{u}_\eps+\eps{\cW})+ \eps \Delta \hat{u}_\eps+
\eps^2 \Delta {\cW}.
\end{array}
\right.
\end{equation}
Observe that for spatially smooth noise on time-scales of order $1$
$$
\hat{u}_\eps=u_\eps + \mathcal{O}(\eps),
$$
is an approximate solution \dbz{of \eqref{main}} for small $\eps$. Furthermore, the main
advantage of the above representation \dbz{as a  system} is based on the fact
that (\ref{CH:sys_sub}) is now a random PDE without stochastic
differentials and all terms appearing are spatially smooth and in time at least H\"older-continuous.
Thus, we can treat all appearing quantities as
functions and analyze the equation path-wisely, i.e., for every
fixed realization of \dbz{the underlying Wiener process} $\cW$.

Therefore, we are able to follow the ideas of the formal derivation presented in
\cite{akk} and derive in the limit the following stochastic Hele-Shaw problem
\begin{equation}\label{H-S-nh}
\left\{
\begin{aligned}
&\Delta v=0\;\mbox{ in }\;\cD \backslash \Gamma(t),\;\;t>0,\\
&\partial_n v=0\;\mbox{ on }\;\partial\cD,\\
&v=\lambda H+\cW \;\mbox{ on }\; \Gamma(t),\\
&V=\frac{1}{2}(\partial_n v^+-\partial_n v^-)\;\mbox{ on }\; \Gamma(t), \\
&\Gamma(0)=\Gamma_0,
\end{aligned}
\right.
\end{equation}
where again $H$ and $V$ are the mean curvature and velocity
respectively of the zero level surface $\Gamma(t)$. For positive $\eps>0$ the domain
$\cD$ admits the following disjoint decomposition
\[
\cD=\cD_{\eps}^+(t) \cup \cD_{\eps}^-(t) \cup \cD^I_{\eps}(t),
\]
where
\[
u_\eps \approx 1 \;\;\;\; \mbox{for} \;\;\;\;  x \in \cD_{\eps}^+(t)
\;\;\;\; \mbox{and}\;\;\;\; u_\eps \approx -1 \;\;\;\; \mbox{for}
\;\;\;\; x \in \cD_{\eps}^-(t).
\]
Moreover, $\cD^I_{\eps}(t)$ is a narrow interfacial region around $\Gamma(t)$ with
thickness of order $\eps$ where $u_\eps$ is neither close to $+1$ nor $-1$.

In particular, we construct an inner solution close to the
interface, and an outer solution away from it. Using the
appropriate matching in orders of $\eps$, we formally pass to the
limit and derive the corresponding free boundary problem. To
avoid additional technicalities, we also assume that the
interface $\Gamma$ does not intersect the boundary. In terms of
simplicity of notation, we drop the subscript $\eps$ in all the
calculations that follow.

\subsection{Outer expansion}
We consider that the inner interface is known\da{,} and seek the
outer expansion far from it, i.e., an expansion in the form
\begin{equation*}
\begin{split}
\hat u & =  \hat u_0+\eps \hat u_1+ \cdots,\\
v & = v_0 +\eps v_1 + \cdots,
\end{split}
\end{equation*}
where `\dbz{+}...' denote higher order terms and $u_0,\dbz{u_1,}\ldots,v_0,\dbz{v_1,}\ldots$
are smooth functions. We insert the outer expansion into the second equation of the  stochastic system
$\eqref{CH:sys_sub}_2$ and obtain
\begin{equation}\label{exp}
\begin{split}
v_0+\eps
v_1+\mathcal{O}(\eps^2)=&-\displaystyle\frac{1}{\eps}(f^{\prime}(\hat
u_0)+\eps
f^{\prime\prime}(\hat u_0)(\hat u_1+\cW)+\mathcal{O}(\eps^2))\\
& +\eps \Delta(\hat u_0+\eps \hat u_1+\mathcal{O}(\eps^2))+\eps^2
\Delta \cW+\mathcal{O}(\eps^2).
\end{split}
\end{equation}
First collecting the \db{terms of order}
$\mathcal{O}\Big{(}\displaystyle\frac{1}{\eps}\Big{)}$   in
\eqref{exp}, we arrive at $$ f^{\prime}(\hat u_0)=0.$$
\db{Thus}\da{,} we get as in Remark 4.1, (1) of \cite{abc}
$$\hat u_0=\pm 1\;.
$$
In the \db{second} step, we collect the $\mathcal{O}(1)$\db{-}terms in \eqref{exp} and
\dbz{derive}
$$v_0=-f^{\prime\prime}(\hat u_0)(\hat
u_1+\cW)+\mathcal{O}(\eps^2).
$$
We plug now the outer expansion into the first equation of  $\eqref{CH:sys_sub}_1$ and
obtain
\begin{equation*}
\partial_t(\hat u_0+\eps \hat u_1+\mathcal{O}(\eps^2))=-\Delta (v_0+\eps
v_1+\mathcal{O}(\eps^2))\;.
\end{equation*}
 As $\hat u_0$ is a constant \dbz{in the outer expansion} we have $\partial_t \hat u_0=0,$ and
 thus\da{,}
collecting the $\mathcal{O}(1)$ terms yields
\begin{equation*}
-\Delta v_0=0 \;.
\end{equation*}
Collecting finally \db{in the third step} all $\mathcal{O}(\eps)$\db{-}terms we arrive at
\begin{equation*}
\partial_t \hat u_1=-\Delta v_1.
\end{equation*}
%
%
\subsection{Inner expansion}
%
%
Let $x$ be a point in $\cD$ that at time $t$ is near the
interface $\Gamma(t)$. Let us introduce the stretched normal
distance to the interface, $z:=\frac{d}{\eps}$, where $d(x,t)$ is
the signed distance from the point $x$ in $\cD$ to the interface
$\Gamma(t)$, such that $d(x,t)>0$ in $\cD_{\eps}^+$ and
$d(x,t)<0$ in $\cD_{\eps}^-$. Obviously $\Gamma$ has the
representation
$$
\Gamma(t)=\{x\in\cD:d(x,t)=0\}.
$$
  If $\Gamma$ is smooth, then $d$ is \dbz{well defined and} smooth near $\Gamma$, and
$|\nabla d|=1$ in a neighborhood of $\Gamma$. Following
\cite{abc} and \cite{pego}, we seek for an inner expansion valid
for $x$ near $\Gamma$ of the form
\[
\begin{array}{lll}
\hat u & = &
q\Big{(}\displaystyle\frac{d(x,t)}{\eps},x,t\Big{)}+\eps Q\Big{(}
\displaystyle\frac{d(x,t)}{\eps},x,t\Big{)} + \cdots,\\
\\
v & = & \tilde q\Big{(}\displaystyle\frac{d(x,t)}{\eps},x,t\Big{)}
+\eps \tilde Q\Big{(}\displaystyle\frac{d(x,t)}{\eps},x,t\Big{)} + \cdots,\\
\end{array}
\]
where \dbz{again} `$+\cdots$' denote higher order terms \dbz{that we neglect} and
$q,Q,\ldots,\tilde{q},\tilde{Q},\ldots$ are \dbz{sufficiently} smooth. It will be
convenient to require that the quantities depending on
$z\;,x,\;t$ are defined for $x$ in a full neighborhood of
$\Gamma$ but do not change when $x$ varies normal to $\Gamma$
with $z$ held fixed, \cite{pego}. We insert the inner expansion
into $\eqref{CH:sys_sub}_2$, utilize that $|\nabla d|^2=1$, and \dbz{thus}
obtain the following expression
\begin{equation}\label{inner}
\tilde q+\eps \tilde Q
+\mathcal{O}(\eps^2)=-\displaystyle\frac{1}{\eps}(f^{\prime}(q)+\eps
f^{\prime\prime}(q)(Q+\cW))+\eps
\Big{(}\displaystyle\frac{\partial_z q}{\eps} \Delta d+
\displaystyle\frac{\partial_{zz} q}{\eps^2}+\partial_z Q \Delta d
+\displaystyle\frac{\partial_{zz} Q}{\eps}\Big{)}+\eps^2 \Delta W.
\end{equation}
We collect the \db{terms of} order
$\mathcal{O}\Big{(}\displaystyle\frac{1}{\eps}\Big{)}$ and derive
\[
\partial_{zz} q-f^{\prime}(q)=0.
\]
By matching now the terms of order $\mathcal{O}(1)$ in
\eqref{inner}, we obtain
\[
\tilde q
=-f^{\prime\prime}(q) Q+\partial_{zz}
Q-f^{\prime\prime}(q) \cW+\partial_z q  \Delta d,
\]
or equivalently
\begin{equation}\label{linop}
\tilde q-\partial_z q \Delta d =
\partial_{zz}Q-f^{\prime\prime}(q)Q- f^{\prime\prime}(q) \cW\;.
\end{equation}
We define the linearized Allen-Cahn operator
\[
{\mathcal L} Q= \partial_{zz}Q-f^{\prime\prime}(q)Q\;.
\]
Then \eqref{linop} is written as
\begin{equation}\label{L}
\tilde q-\partial_z q \Delta d = {\mathcal L}
Q-f^{\prime\prime}(q) \cW\;.
\end{equation}
This equation is solvable if for any
$\chi \in \text{Ker}( {\mathcal L}^*)$ it holds that
$
\chi \perp (\tilde q-\partial_z q \Delta d+f^{\prime\prime}(q)
\cW),
$
or equivalently if
\begin{equation}\label{integr}
\int_{-\infty}^{\infty} \chi  \cdot(\tilde q-\partial_z q
\Delta d+f^{\prime\prime}(q) \cW) \, dz =0.
\end{equation}
Obviously, for any $x$ on $\Gamma$ it holds that $d(x,t)=0$ and $\Delta
d(x,t)=H(x,t)$. Replacing in \eqref{integr} we obtain the following sufficient
condition on the interface $\Gamma$:
\begin{equation}\label{tildeqH}
\tilde q= \lambda H-f^{\prime\prime}(q) \cW.
\end{equation}
Plugging the inner expansion into  $\eqref{CH:sys_sub}_1$ we obtain
\begin{equation}\label{innerexp}
\displaystyle\frac{\partial_z q}{\eps}d_t+\partial_z  Q
d_t+\mathcal{O}(\eps)= - \Big( \displaystyle\frac{\partial_z
\tilde q}{\eps} \Delta d+ \displaystyle\frac{\partial_{zz} \tilde
q}{\eps^2}+\partial_z \tilde Q \Delta d
+\displaystyle\frac{\partial_{zz} \tilde Q}{\eps}\Big).
\end{equation}
We collect the terms of order
$\mathcal{O} ( \eps^{-2} )$
and arrive at
\[
\partial_{zz} \tilde q =0,
\]
which implies that for some functions $a$ and $b$
\[
\tilde q = a(x,t) z+b(x,t).
\]
To proceed further, the matching condition for the inner and
outer expansions must be developed. In general, these are
obtained by the following procedure (\dbz{see }\cite{cf}). Fixing $x \in \Gamma$,
we seek to match the expansions by requiring formally for $z\to \infty$
\[
\tilde q +\eps \tilde Q+\mathcal{O}(\eps^2) =v_0+\eps
v_1+\mathcal{O}(\eps^2),
\]
and thus in order $\mathcal{O}(1)$
\[
v_0=\displaystyle\lim_{z \rightarrow \infty} \tilde q =
\displaystyle\lim_{z \rightarrow \infty}(a(x,t) z+b(x,t)).
\]
We obtain $a=0$ and thus,
$\tilde{q}=b$.   Hence, utilizing \eqref{tildeqH} we
have that on the interface
\[
v_0=\lambda H-f^{\prime\prime}(q) \cW
=\lambda H+ \cW\;,
\]
where we used that $q$ solves the Euler-Lagrange equation
\begin{equation*}
\begin{split}
&-q''(z)+f'(q(z))=0,\;\;z\in \mathbb{R},\\
&\displaystyle{\lim_{z\rightarrow\pm\infty}}q(z)=\pm 1,\;\;q(0)=0,
\end{split}
\end{equation*}
while on the inner interface with $z=d/\eps=0$ we have $f''(q)=3q^2-1=-1$ since $q(0)=0$.

 What is still missing is the evolution law, which should come from
the inner expansion. From \eqref{innerexp}, we collect the terms of order
$\mathcal{O}(1/\eps)$ and
obtain 
\[
\partial_z q  d_t= -\partial_z \tilde q \Delta d - \partial_{zz} \tilde
Q \;.
\]
Recall that $-d_t=V,$ while $\Delta d=H$ \dbz{(see for example \cite{abc})} and
integrate over $z$ from $-\infty$ to $\infty$ to \dbz{derive}
\[
-\displaystyle\int_{-\infty}^{\infty} \partial_z q V \, dz = -
\displaystyle\int_{-\infty}^{\infty} \partial_{zz} \tilde Q  \,
dz\;.
\]
From the matching conditions we get
\[
q(+\infty)=1
\quad\text{and}\quad
q(-\infty)=-1\;.
\]
Hence, we have
\[
V=\displaystyle\frac{1}{2} [ \partial_z \tilde
Q(+\infty)-\partial_z \tilde Q(-\infty) ]\;.
\]
Thus the stochastic Hele-Shaw problem \eqref{H-S-nh} is established
formally as the sharp interface limit, in the case $\sigma=1$.

\begin{remark}
In the case $\sigma>1$, we follow the same construction of inner
and outer solutions as above and obtain in the limit, the
deterministic Hele-Shaw problem, (\cite{abc})
\[
\left\{
\begin{aligned}
&\Delta v=0\;\mbox{  in  }\;\cD \backslash \Gamma(t),\;\;t>0,\\
&\partial_n v=0\;\mbox{  on  }\;\partial\cD,\\
&v=\lambda H\;\mbox{  on  }\; \Gamma(t),\\
&V=\frac{1}{2}(\partial_n v^+-\partial_n v^-)\;\mbox{  on  }\; \Gamma(t), \\
&\Gamma(0)=\Gamma_0,
\end{aligned}
\right.
\]
where $H$ and $V$ are the mean curvature and velocity
respectively of the zero level surface $\Gamma(t)$ contained in
the interfacial region $\cD^I_{\eps}(t)$.
\end{remark}
\begin{remark}
Note that the change of variables
$$\hat{u}_\eps:=u_\eps - \eps^\sigma{\cW},$$ implies that
$$\hat{u}_\eps:=u_\eps + \mathcal{O}(\eps^\sigma).$$ Thus, only if $\sigma\geq
1$, $\hat{u}$ is permitted to be expanded as in the presented
inner expansion using $q$. The key difference to the deterministic analysis
is that in the nonlinearity we have $\frac1{\eps}f(\hat{u}_\eps +\eps^\sigma \cW)$.
In case $\sigma >1$ there is no contribution of $\cW$  in an asymptotic expansion in terms of order $\CO(1)$ and $\CO(1/\eps)$,
while for $\sigma=1$ there is an impact of $\cW$ on terms of order $\CO(1)$.
\end{remark}
\begin{remark}
When $0\leq\sigma<1$ the strategy presented in this section
fails. For this case, we might think of avoiding the change of
variables and apply the formal asymptotics of \cite{akk} directly
instead. \db{If we split the noise term $\dot{\mathcal{W}}$ in
\eqref{main}, we obtain  the following stochastic system with two
noise sources \dbz{(recall that $\varepsilon^\sigma
\dot{\mathcal{W}}=\Delta \dot{V}+\dot{W}$)}
\begin{equation*}
\left\{
\begin{array}{lll}
\partial_t u & = & -\Delta v+\dot{W},\\
\\
v & = & -\displaystyle\frac{f^{\prime}(u)}{\eps}+\eps \Delta
u+\dot{V},
\end{array}
\right.
\end{equation*}
}
The sharp interface
limit \db{should coincide} to \eqref{nhH-S*}, but for
$\displaystyle{\lim_{\eps\rightarrow 0^+}}G_1$,
$\displaystyle{\lim_{\eps\rightarrow 0^+}}G_2$ replaced by
$$\displaystyle{\lim_{\eps\rightarrow
0^+}}\dot{W}(\cdot,\eps),
\;\;\;\;{\rm{and}}\;\;\;\;\displaystyle{\lim_{\eps\rightarrow
0^+}}\dot{V}(\cdot,\eps),$$ respectively. When $0<\sigma<1$ we would obtain that both these limits are $0$
and thus the limiting problem is a deterministic Hele-Shaw problem without the contribution of the noise.
But this is a very dangerous reasoning, as \dbz{the} noise terms \dbz{$\dot{V}$ and $\dot{W}$}, even if they \dbz{would be} $\eps$-independent are not of order $\CO(1)$,
and we would still expect an impact of the noise terms on the limiting problem.
\end{remark}


\section{The sharp interface limit}
\label{sec:main}


The main result of this paper is that \eqref{main}, as $\eps$
tends to zero, may have a deterministic or a stochastic profile
depending on the strength of the additive noise in terms of
$\eps$. Only large noise perturbations with $\sigma=1$ generate a
stochastic limit problem. Here\da{,} we discuss the limit for
smaller noise strength.

Let us first precisely state our problem. We 
assume that noise is induced by the
formal derivative of a $Q$-Wiener process $\cW$ in a Fourier
series representation (see \cite{DPZa:14}); for simplicity, the
only $\eps$-dependence will appear in the noise strength, and
thus, for the rest of this paper we shall use the notation
$\eps^\sigma d{\cW}(x,t)$ for the additive noise, where
$\sigma\in \mathbb{R}$ \dbz{is a scaling parameter}.

\begin{assumption}
\label{ass:W}
Let $\cW$ be a $Q$-Wiener process \db{on some probability space $(\Omega,\mathcal{A},\mathbb{P})$} such that
\[
\cW(t)=\sum_{k\in\N}\alpha_k\beta_k(t)e_k,
\]
for an orthonormal basis $(e_k)_{k\in\N}$ \db{in $L^2(\mathcal{D})$}, independent
real-valued Brownian motions $(\beta_k)_{k\in\N}$ \db{on $(\Omega,\mathcal{A},\mathbb{P})$}, and \db{real-valued}
coefficients $\alpha_k$ such that $Qe_k=\alpha_k^2  e_k$.
Furthermore, we assume that the noise \db{has some weak smoothness} in
space, i.e., $Q$ satisfies
\begin{equation}
\label{e:trace}
 \text{\rm trace}(\Delta^{-1}Q) <\infty.
\end{equation}
To deal with a mass-conserving stochastic problem, we impose the
condition
 $$\int_{\cD}\cW(t) dx =0\;.
 $$
\end{assumption}
Note that \eqref{e:trace} implies that the Wiener-process $\cW(t)$ is $H^{-1}\db{(\mathcal{D})}$-valued.
This is the minimal requirement for the approximation
theorem presented in the sequel; we might need more regularity,
in order to have the stochastic Hele-Shaw limit problem well
defined, or while performing the formal asymptotics.

\db{Recall \eqref{main} in It\^{o}-formulation}
\begin{equation}\label{CH}
d u_\eps=\Delta
(-\eps \Delta u_\eps+\eps^{-1}f^{\prime}(u_\eps)) dt + \eps^\sigma d{\cW}(x,t),
\end{equation}
associated to Neumann conditions on the boundary \dbz{of $\mathcal{D}$, so that the equation }\db{is still mass conservative.}

The following theorem \db{for the existence of solutions} is well known. See for example \cite{DPDe:96}.
\begin{theorem}
 Let $\mathcal{D}$ be a rectangle in dimensions $1,2,3$. If
 $Q=I$ or $trace(\Delta^{-1+\delta}Q)<\infty$ for
 $\delta>0$, then the following holds true:
 \begin{enumerate}
 \item
 if $u_0$ is in $H^{-1}(\mathcal{D})$, there exists a unique solution for
 the problem \eqref{main} in $C([0,T];H^{-1}(\mathcal{D}))$,
\item
if $u_0$ is in $L^{2}(\mathcal{D})$, then the solution for
 the problem \eqref{main} is in
 $L^\infty(0,T;L^{2}(\mathcal{D}))$.
 \end{enumerate}
\end{theorem}
Note that the previous theorem could be extended for general
Lipschitz domains in dimensions $2$ and $3$ under some additional
assumptions of minimum eigenfunctions growth, cf. the arguments
in \cite{ak}. For the analysis underlying our results \db{and to
avoid technicalities}, we will for the remainder of the paper
always assume:
\begin{assumption}
\label{ass:IC}
\db{For any initial condition in $H^{-1}(\mathcal{D})$
there} exists a unique solution for the problem \eqref{main} in $C([0,T];H^{-1}(\mathcal{D}))$,
\db{which is sufficiently regular such that we can apply It\^o-formula to the $H^{-1}$-norm.}
\end{assumption}
\db{The regularity required for this assumption is usually straightforward to verify.
For example \cite{DPDe:96} applies the It\^o-formula to a \dbz{finite dimensional} spectral Galerkin-approximation and then passes to the limit.}
Introducing the chemical potential $v_\eps$, the equation is as in the formal derivation
rewritten as a stochastic system. Indeed,
\db{
\[
 \text{for } T>0 \text{ let }
\mathcal{D}_T:=\mathcal{D}\times(0,T),
\]
}
then \eqref{CH} is written as
\begin{equation}\label{CH:sys}
\begin{split}
d u_\eps  = & -\Delta v_\eps dt+ \eps^\sigma
d{\cW}\;\;\;\;\mbox{in}\;\;\mathcal{D}_T,
\\
v_\eps  = & -\displaystyle\frac1{\eps}{f^{\prime}(u_\eps)}+ \eps
\Delta u_\eps\;\;\;\;\mbox{in}\;\;\mathcal{D}_T,
\end{split}
\end{equation}
subject to Neumann boundary conditions
$$\frac{\partial u_\eps}{\partial n}=\frac{\partial \Delta
u_\eps}{\partial n}=0
\dbz{\qquad\text{on }\partial\mathcal{D}\times(0,T)}\;.
$$
Our main analytic theorem  considers \textit{a sufficiently small
noise resulting to a deterministic sharp interface} limiting
behavior. In particular, we  analyze the case
$$
\sigma \gg \sigma_0=1,
$$
where $\sigma_0$ is the \dbz{conjectured} borderline case,
where according to our formal calculation the noise has an impact on the limiting model.
Under some assumptions on the initial condition $u_\eps(0)$, the
limit of $u_\eps$ and $v_\eps$ as $\eps\rightarrow 0$ solves
\db{the deterministic Hele-Shaw
problem on a given time interval $[0,T]$.}
We will state the precise formulation of this argument in
Theorem \ref{thm:main}, and then present the rigorous proof.

\db{While our rigorous result can only treat very small noise, }
the formal
derivation of Section \ref{formal} motivates the following conjecture
implying a \textit{stochastic} sharp interface limit:
\begin{conj}
For $\sigma=1$ the limit of $u_\eps$
and $v_\eps$ solves the stochastic Hele-Shaw problem
\begin{equation}\label{stochHS} \left\{
\begin{aligned}
&\Delta v=0\;\mbox{ in }\; \mathcal{D} \backslash\Gamma(t),\;\;t>0, \\
&\partial_n v=0\; \mbox{ on }\;\partial\mathcal{D},\\
&v=\lambda H + \cW \;\mbox{ on }\; \Gamma(t),\\
&V=\frac{1}{2}(\partial_n v^+-\partial_n v^-)\;  \mbox{ on }\; \Gamma(t), \\
&\Gamma(0)=\Gamma_0.
\end{aligned}
\right.
\end{equation}
\end{conj}
\begin{remark}
In \dbz{S}ection 2 by formal asymptotics,  we only presented an indication for the correctness of the
conjecture. A rigorous proof of this conjecture
remains open at the moment. We hope to attack the problem to its
full generality in the near future.
\end{remark}

\begin{remark} Note that $\cW$ is a Wiener process,
and the equation $v=\lambda H + \cW$ on $\Gamma(t)$ appearing in
(\ref{stochHS}), has a rigorous mathematical meaning in terms of
functions. In fact, no noise is present, while a random  equation
appears on $\mathcal{D} \backslash\Gamma(t)$ in the following
sense. For any given $t$, $\Gamma(t)$ is defined by its velocity
$V$ and thus is known, and the unknown function $v$ on
$\Gamma(t)$ is a stochastic process. Thus, the problem for fixed
$t$ is posed in between the inner boundary $\Gamma=\Gamma(t)$ and
the outer boundary $\partial\mathcal{D}$ as follows
\begin{equation}
\Delta v=0\;\mbox{ in }\; \mathcal{D} \backslash\Gamma, \qquad
\partial_n v=0\; \mbox{ on }\;\partial\mathcal{D}, \qquad
v=\lambda H + \cW \;\mbox{ on }\; \Gamma,
\end{equation}
the inner boundary condition $v|_\Gamma$ being a realization of a
$t$-dependent stochastic process.
\end{remark}

 \subsection{Statement of the Main Theorem}

 In this section, we shall state the main analytic theorem of this paper,
 concerning the sharp interface limiting profile
 for sufficiently small noise strength.
To approximate the stochastic solution we  use the same approximations
 $u_\eps^A$ and $v_\eps^A$ as in \cite{abc} proposed in the absence of noise.
 For a precise definition see further below. In our proof we
 follow the ideas of the proof of their Theorem 2.1, and need to adapt the
 analysis to \da{ the} presence of noise.

The main difference \da{ concerns} the noise in the equation
for the residual
\begin{equation}
\label{def:R}
 R:=u_\eps-u_\eps^A.
\end{equation}
\db{ We show bounds for this error in our main Theorem
\ref{thm:main} below, while in Remark \ref{rem1} we comment on
the smallness of the noise necessary for the main result.
Unfortunately, we are quite far away from treating \dbz{the case} $\sigma\da{>}1$
close to $1$. Another problem is the weak estimate on the
spectral stability of the linearized operator stated in
Proposition \ref{abc:spec}. See also Remark \ref{rem2}. }

\begin{assumption}
 \label{ass:Gamma}
 Let the family  $\{\Gamma(t)\}_{t\in[0,T]}$ of smooth closed hypersurfaces
 together with the functions $\{v(t)\}_{t\in[0,T]}$
be a solution
of the deterministic Hele-Shaw problem (i.e., equation \eqref{stochHS} with $\cW=0$)
such that the interfaces do not intersect with the boundary $\partial\cD$, i.e.,\
$\Gamma(t)\subset\cD$ for all $t\in[0,T]$.
\end{assumption}
With $\Gamma$ from Assumption \ref{ass:Gamma} the  \dbz{authors in \cite{abc}}
construct a pair of approximate solutions
$(u_\eps^A,\;v_\eps^A)$,
so that $\Gamma(t)$ is the
zero level set of $u_\eps^A(t)$,  which satisfies
\begin{equation}\label{appr}
\begin{split}
d u_\eps^A  = & -\Delta v_\eps^A dt\;\;\;\;\mbox{in}\;\;\mathcal{D}_T,
\\
v_\eps^A  = & -\displaystyle\frac1{\eps}{f^{\prime}(u_\eps^A)}+
\eps \Delta u_\eps^A+r_\eps^A\;\;\;\;\mbox{in}\;\;\mathcal{D}_T,
\end{split}
\end{equation}
for boundary conditions
$$
\frac{\partial u_\eps^A}{\partial n}=\frac{\partial \Delta
u_\eps^A}{\partial n}=0\;\; \quad \text{on } \partial\mathcal{D}.
$$
We recall that
$u_\eps^A$ approximates the deterministic version of equation
\eqref{main} \dbz{(i.e., for $\cW=0$).}
The error term  $r_\eps^A$ is
bounded in terms of $\eps$, and depending on the smoothness of
$\Gamma$ and the number of approximation steps, the bound on $r_\eps^A$ can be
arbitrarily small. For details see  relation (4.30) and Theorem
4.12 in \cite{abc}.

We will summarize the results of \cite{abc} that we need for our proof in  the following Theorem.
\begin{theorem}
\label{abc:approx} Under the Assumption \ref{ass:Gamma}, for any
$K>0$ there exists a pair  \db{$(u_\eps^A,\;v_\eps^A)$} of
solutions to \eqref{appr}, such that
\[
\|r_\eps^A\|_{C^0(\cD_T)} \leq C\eps^{K-2}\;.
\]
Moreover, it holds that
\[
\|v_\eps^A - v\|_{C^0(\cD_T)} \leq C\eps\;,
\]
and finally for $x$
away from $\Gamma(t)$ \mbox{(}i.e., $\text{d}(x,\Gamma(t)) \geq c \eps$\mbox{)}
\[
|u_\eps^A(t,x)- 1| \leq C\eps \quad\text{or}\quad   |u_\eps^A(t,x)+ 1| \leq C\eps
.
\]
\end{theorem}
We present now the following spectral estimate, useful in our
proof; we refer to \cite{Ch:94} for dimensions larger than two,
and to \cite{AlFu:93} for dimension two. Unfortunately, this \dbz{estimate} is also the key problem
to extend the approximation result beyond time-scales of order $1$.
\begin{prop}[Proposition 3.1 of \cite{abc}]
\label{abc:spec} Let $u_\eps^A$ be the approximation given in
Theorem \ref{abc:approx}. Then for all $w\in H^1(\cD)$ satisfying
Neumann boundary conditions such that $\int_{\cD}wdx=0$, the following
estimate is valid
\[
 \int_{\cD}  [\eps |\nabla w|^2+\frac1\eps f''(u_\eps^A) w^2] dx \geq - C_0\|\nabla
 w\|^2_{L^2}.
\]
\end{prop}
\db{Our main theorem provides bounds for the residual $R$:}
\begin{theorem}
\db{
 \label{thm:main}{\rm \textbf{(Main Theorem)}}
 Let Assumption \ref{ass:W} for the noise and \ref{ass:IC} for the existence of solutions be true.
 }

\db{
Fix a time $T>0$ and any $p\in(2,3]$, a radius $\eps^\gamma$ with
\[
\gamma > \frac1{p-2} \Big[ 1+ \frac{2p+d(p-2)}{2p-d(p-2)}\cdot \frac{p+2}{p}  \Big]
\]
and a noise strength $\eps^\sigma$ with
\[
\sigma > \gamma + \frac{2p+d(p-2)}{2p-d(p-2)}\cdot \frac{p+2}{p}\;.
\]
}

 \db{Then for some small $\kappa>0$ and a generic constant $C>0$ there is for all large $\ell>0$ a constant $C_\ell$ such that the following is true:}

 \db{For all
 families  $\{\Gamma(t)\}_{t\in[0,T]}$ of closed hypersurfaces satisfying Assumption \ref{ass:Gamma},
with  corresponding  approximation  $u_\eps^A$ and  $v_\eps^A$ defined above in Theorem \ref{abc:approx}.
and  $u_\eps$ solution of the stochastic Cahn-Hilliard
 equation \eqref{CH} such that $u_\eps(0)=u_\eps^A(0)$,
 the following probability estimates hold:
 \[
 \mathbb{P} \left( \|R \|_{L^p([0,T]\times\cD)}
  \leq  \eps^\gamma \right)   \geq 1-C_\ell\eps^\ell  \;,
\]
\[
  \mathbb{P} \left( \|R \|^2_{L^\infty (0,T,H^{-1})}
  \leq  C[ \eps^{p\gamma-1}
+  \eps^{\sigma+\gamma-\kappa} ]\right)   \geq 1-C_\ell\eps^\ell \;,
\]
and
\[
  \mathbb{P} \left(  \|R \|^2_{L^2 (0,T,H^{1})}
  \leq  C[ \eps^{-1-2/p + 2\gamma} +
\eps^{-1+\sigma+\gamma-\kappa} ]\right)   \geq 1-C_\ell\eps^\ell  \;,
\]
where $R:=u_\eps-u_\eps^A$ is the error defined in \eqref{def:R}.
}
\end{theorem}

\begin{remark}
\label{rem1}
Let us remark that in dimension $d=2$ one can easily check
that we obtain the smallest possible value both for $\sigma$ and $\gamma$ for $p=3$.
In that case $\gamma>6$ and $\sigma> 23/3$. This is in well agreement with the $\gamma$ derived in  \cite{abc},
but unfortunately we can only consider very small noise strength. But it seems that
using the $H^{-1}$-norm and spectral information available there is no improvement possible.

For dimension $d=3$ again the noise strength is small, but the result is not that clear.
While the smallest value for $\gamma$ is still attained at $p=3$ (with $\gamma>6$ and $\sigma>11$)
we obtain the smallest value of $\sigma$ for some $p<3$.
\end{remark}
\begin{remark}\db{
Let us remark on the fact that we take $R(0)=0$. We could take more general initial conditions $u_\eps(0)$ for the Cahn-Hilliard equation.
Looking closely into the proof of Theorem \ref{thm:main} we could allow for \dbz{an initial error $R(0)$ with} $\|R(0)\|^2_{H^{-1}}= \mathcal{O}(\eps^{p\gamma-1}+\eps^{\sigma+\gamma-\kappa})$.
}
\end{remark}
\begin{remark}\label{rem2}
Let us state two main problems with the approach presented.

First, the spectral estimate in Theorem \ref{abc:spec} yields an unstable
eigenvalue of order $\mathcal{O}(1)$. This immediately restricts
any approximation result to time scales of order $\mathcal{O}(1)$.
But we strongly believe that this eigenvalue represents only a motion of
the interfaces itself. One would need spectral information orthogonal
to the space of all possible approximations $u_\eps^A$, which are
parametrized by the hypersurfaces $\Gamma$. But this does not seem to be available at the moment.

Moreover, later in the closure of the estimate we can only allow
$\sigma >\sigma_0$ large enough, i.e., for sufficiently small
noise strength. Here, an additional problem is that the $H^{-1}$-norm is not
strong enough to control the nonlinearity, and from the spectral
theorem \ref{abc:spec}, we do not get any higher order norms that would help in
the estimate. Nevertheless, if we start with higher order norms like $L^2$, for instance, then
there are no spectral estimates available at all.
 \end{remark}


\section{The proof of \db{the} Main Theorem \ref{thm:main}}
\label{sec:proof}
%
\subsection{Idea of Proof}

For the proof we define for $p\in(2,3]$ and
$\sigma>\gamma>0$ (both fixed later), the stopping time
\begin{equation}
\label{e:defTeps}
 T_\eps := \inf\Big{\{}t\in[0,T]:\  \Big(\int_0^t\|R(s)\|^p_{L^p}ds \Big)^{1/p}  > \eps^\gamma
 \Big{\}},
\end{equation}
where the convention is that $T_\eps=T$ if the condition is never true.

The general strategy for the proof of the main theorem is the
following:
\begin{enumerate}
\item Use It\^o-formula for
$d\|R\|^2_{H^{-1}}$.
\item Consider all estimates up to $T_\eps$ only.
\item  Bound the stochastic integrals (at least on a set with high
probability).
\item Show that $T_\eps=T$ with high probability using the bound derived for
$\int_0^t\|R\|_{L^p}dt$ up to $T_\eps$.
\end{enumerate}


\subsection{A differential equation for the error}


Let us first derive an SPDE for $R$ from \eqref{def:R}, using
\eqref{appr} and \eqref{CH:sys}, as follows
\begin{equation}\label{e:SPDER}
 \begin{split}
d R = du_\eps- d u_\eps^A
&=  \Delta v_\eps^A dt-\Delta v_\eps dt + \eps^\sigma d{\cW}\\
&=  \Big[-\frac1{\eps} \Delta{f^{\prime}(u_\eps^A)}+ \eps \Delta^2 u_\eps^A
 +\Delta r_\eps^A + \frac1{\eps}\Delta{f^{\prime}(u_\eps)}- \eps\Delta^2 u_\eps\Big] dt + \eps^\sigma d{\cW}
\\
 &=\frac1{\eps}\Big[\Delta{f^{\prime}(u_\eps^A+R)}- \Delta{f^{\prime}(u_\eps^A)}\Big] dt +
 \Big[ - \eps \Delta^2R
 +\Delta r_\eps^A \Big] dt + \eps^\sigma d{\cW}.
\end{split}
\end{equation}

\subsection{The $H^{-1}$-norm of $R$.}

The approximate solutions $u_\eps^A$ and $v_\eps^A$ are, by their
construction, functions in $C^2(\overline{\mathcal{D}_T})$, while
$u_\eps^A$ satisfies for all $t\in[0,T]$
$$
\int_\mathcal{D} u_\eps^A(t) dx=0\;.
$$
Since (\ref{CH}) is mass conservative,  we can conclude that mass
conservation also holds for $R$, i.e.,\ for all $t\in[0,T]$
$$
\int_\mathcal{D} R(t) dx=0\;.
$$
Observe that the operator $-\Delta$ is a symmetric positive operator on the space
\[ \mathcal{H}^2:=
\Big{\{}w\in C^2(\overline{\mathcal{D}}) : \int_\mathcal{D} w\;
dx=0 \quad\text{and}\quad
\partial_n w = 0 \text{ on }
\partial\mathcal{D}  \Big{\}}.
\]
Therefore, by elliptic regularity, the operator $-\Delta
:\mathcal{H}^2 \to L^2$ is bijective. So, we  can invert it and
 for any $t\in[0,T]$ there exists a unique $\psi(t) \in
\mathcal{H}^2$ such that
\begin{equation}
\label{def:psi} -\Delta\psi(t)=R(t), \quad \text{or
equivalently }\quad (-\Delta)^{-1}  R(t)= \psi(t).
\end{equation}
With the scalar
product $\langle \cdot , \cdot \rangle$  in $L^2$
the $H^{-1}$-norm of $R$ is given by
\[
\|R\|^2_{H^{-1}} =   \|(-\Delta)^{-1/2}R\|^2_{L^2} =
\|(-\Delta)^{1/2}\psi\|^2_{L^2} =\|\nabla \psi\|^2_{L^2}  =
\langle \psi , R \rangle.
\]
Since
$$
\langle d\psi , R \rangle = \langle  -\Delta d R , R
\rangle
 =  \langle   d R , -\Delta R \rangle,
 $$
 considering the
 It\^o-differential,
 we obtain
\begin{equation}
 \label{e:dRH-1}
\begin{split}
 \tfrac12 d \|R\|^2_{H^{-1}}
& = \langle \psi , d R \rangle  +  \tfrac12 \langle d\psi , d R \rangle
 = \langle \psi , d R \rangle +  \tfrac12 \eps^{2\sigma} \langle  (-\Delta)^{-1} d\mathcal{W} , d  \mathcal{W} \rangle \\
& =  \langle \psi , d R \rangle +  \tfrac12 \eps^{2\sigma}
\text{tr}(Q^{1/2} (-\Delta)^{-1}Q^{1/2}).
\end{split}
\end{equation}
Here, by Assumption \ref{ass:W} the trace in the previous estimate is bounded.
So, using \eqref{CH:sys} and \eqref{appr}, we arrive at
\begin{equation}\label{e:d_psiR1}
\begin{split}
\langle \psi , d R \rangle
&=\langle \psi , d (u_\eps-u_\eps^A)\rangle
=\langle \psi , (-\Delta)(v_\eps-v_\eps^A)dt + \eps^\sigma d\mathcal{W}\rangle \\
&=\langle R , (v_\eps-v_\eps^A)\rangle dt
+  \eps^\sigma \langle \psi , d\mathcal{W}\rangle.
\end{split}
\end{equation}
Using again \eqref{CH:sys} and \eqref{appr} in order to replace
the $v$'s, yields  the following equality
\begin{equation}
\label{e:d_psiR2}
\begin{split}
\langle \psi , d R \rangle
&=- \eps^{-1}  \langle R ,f'(u_\eps)-f'(u^A_\eps)\rangle dt
+ \eps  \langle R , \Delta(u_\eps- u^A_\eps) \rangle dt
-   \langle R , r_\eps^A\rangle dt
+  \eps^\sigma \langle \psi , d\mathcal{W}\rangle.\\
&= - \eps^{-1}\langle R ,(f'(u_\eps)-f'(u^A_\eps))\rangle dt
- \eps  \| \nabla R \|^2 dt
-   \langle R , r_\eps^A\rangle dt
+  \eps^\sigma \langle \psi , d\mathcal{W}\rangle.\\
\end{split}
\end{equation}
\db{
\begin{definition}
For any positive integer $p$, we define the $L^p$-norms
$$\|f\|_{p,\mathcal{D}}:=\Big{(}\int_{\mathcal{D}}|f|^pdx\Big{)}^{1/p}
\qquad\text{and}\qquad
\|f\|_{p,\mathcal{D}_t}:=\Big{(}\int_0^t\int_{\mathcal{D}}|f|^pdxds\Big{)}^{1/p}.
$$
\end{definition}
}
Also, \db{we denote} by $\|\cdot\|$ the
usual $L^2(\mathcal{D})$-norm and by  $\|\cdot\|_{L^p}$
the $L^p(\mathcal{D})$-norm.

Applying Taylor's formula to expand $f'(u_\eps)$ around
$u_\eps^A$, with residua\db{l} $\mathcal{N}(u_\eps^A,R)$, we
have
\[
f'(u_\eps)-f'(u_\eps^A)=f''(u_\eps^A)R+\mathcal{N}(u_\eps^A,R)\;.
\]
The crucial bound for the nonlinearity in the residual is the
following result from Lemma 2.2 of \cite{abc}. It is based on a
direct representation of the remainder $\mathcal{N}$ in the
Taylor expansion together with the fact that $u_\eps^A$ is
uniformly bounded.
\begin{lemma}
 Let $p\in(2,3]$ and $q$ such that $\frac{1}{p}+\frac{1}{q}=1$, then it holds that
\begin{equation}
\label{thm8}
-\int\eps^{-1}\mathcal{N}(u_\eps^A,R)R\leq
c\eps^{-1}\|R\|_{p,\mathcal{D}}^p.
\end{equation}
\end{lemma}
Thus, we obtain
\begin{equation}
\label{e:boundNL}
\begin{split}
 -\frac1\eps\langle R ,(f'(u_\eps)-f'(u^A_\eps))\rangle
&= -\frac1\eps\langle R , f''(u_\eps^A)R \rangle
-  \frac1\eps\langle R , \mathcal{N}(u_\eps^A,R) \rangle \\
&\leq  -\frac1\eps\langle R , f''(u_\eps^A)R \rangle +
c\eps^{-1}\|R\|_{p,\mathcal{D}}^p\;.
\end{split}
\end{equation}
Relations (\ref{e:dRH-1}), (\ref{e:d_psiR2}) and (\ref{e:boundNL})
yield the following first key estimate
\begin{equation}
\label{e:apKEY}
\begin{split}
\tfrac12 d \|\nabla \psi \|^2
& +  \eps  \| \nabla R \|^2 dt
+ \frac1\eps\langle R , f''(u_\eps^A)R \rangle dt
\\& \leq
  c\eps^{-1} \|R\|_{p,\mathcal{D}}^p dt
+  \| R \|_{p,\cD} \|r_\eps^A\|_{q,\cD} dt +  \eps^\sigma \langle
\psi , d\mathcal{W}\rangle +   C_\text{tr} \eps^{2\sigma} dt.
\end{split}
\end{equation}
From this \textit{a-priori} estimate, we now can derive a uniform
bound for $\|\nabla \psi\|$ and later a mean square bound on $\| \nabla R \|$.
Both estimates  still involve the
$L^p$-norm of $R$ on the right hand side, and we use the stopping time $T_\eps$ to control this.
%
%
%
\subsection{Technical Lemmas}
%
%
%
We first need the following Lemma of Burkholder-Davis-Gundy type for stochastic integrals.
Recall the stopping time $T_\eps$
from (\ref{e:defTeps}).
\begin{lemma}
\label{lem:boundSI}
 Let $f$ be a continuous real valued function, and $\Delta\psi=R$ as before.
Then for all $\kappa>0$, $\ell>1$ there exists a constant
$C=C(\ell,T,\kappa)$ such that
\[
\mathbb{P}\Big( \sup_{[0,T_\eps]}|\int_0^tf\langle \psi , d\cW
\rangle| \geq \eps^{\gamma-\kappa} \Big ) \leq
C\eps^{\ell\kappa}\|f\|^\ell_{L^{2p/(p-2)}}.
\]
\end{lemma}
\begin{proof}
We shall use the Chebychev's inequality. Thus, we need to bound
the moments first. Applying Burkholder-Davis-Gundy inequality
(using that $\Delta^{-1}Q\Delta^{-1}$ is a bounded operator by
assumption), we obtain
\begin{equation*}
\begin{split}
  \mathbb{E}\sup_{[0,T_\eps]}\Big|\int_0^tf(s)\langle \psi(s) , d\cW(s) \rangle\Big|^\ell
\leq& C_\ell \mathbb{E} \Big|\int_0^{T_\eps} f^2(s)\langle \psi(s) , Q \psi(s) \rangle ds\Big|^{\ell/2} \\
\leq& C \mathbb{E} \Big|\int_0^{T_\eps} f^2(s) \|R(s)\|_{L^2}^2 ds \Big|^{\ell/2} \\
\leq& C   \mathbb{E} \Big|
\Big( \int_0^{T_\eps} \|R(s)\|_{L^p}^p ds \Big)^{2/p} \Big|^{\ell/2} \Big|\int_0^{T_\eps} f^{2p/(p-2)}\Big|^{(p-2)\ell / (2p)}\\
\leq& C \|f\|^\ell_{L^{2p/(p-2)}}    \eps^{\gamma \ell}.
\end{split}
\end{equation*}
Here, we  applied H\"older's inequality and the definition of
$T_\eps\leq T$.

Furthermore, using Chebychev's inequality, we obtain the result
as follows
\[\mathbb{P}\Big( \sup_{[0,T_\eps]}|\int_0^tf\langle \psi , d\cW \rangle| \geq \eps^{\gamma-\kappa} \Big )
 \leq \eps^{-\ell(\gamma-\kappa)} \mathbb{E}\sup_{[0,T_\eps]}\Big|\int_0^tf\langle \psi , d\cW \rangle\Big|^\ell
\leq C \eps^{\ell\kappa}\|f\|^\ell_{L^{2p/(p-2)}}.
\]
\end{proof}
Now we present the following stochastic version of Gronwall's Lemma.
\begin{lemma}
\label{lem:stochGW} Let $X$, $\mathcal{F}_i$, $\lambda$ be real
valued processes, and $\mathcal{G}$ be a Hilbert-space valued
one. Furthermore, assume that
\[
dX:=[\lambda X+\mathcal{F}_1]dt + \langle \mathcal{G},d
\mathcal{W}\rangle,
\]
and that
$$
\mathcal{F}_1 \leq \mathcal{F}_2.
$$
Then the following
inequality holds true
\[
X(t) \leq \dbz{e^{\Lambda(t)}X(0)} + \int_0^t
e^{\Lambda(t)-\Lambda(s)}{\da{\mathcal{F}_2}}(s)ds +  \int_0^t
e^{\Lambda(t)-\Lambda(s)}\langle \mathcal{G}(s),d
\mathcal{W}(s)\rangle,
\]
for $$\Lambda(t):=\int_0^t \lambda(s)ds.$$
\end{lemma}
\begin{proof}
We define $$Y(t):=X(t)e^{-\Lambda(t)}.$$ By the definition of
the process $Y$, we obtain easily
\[
dY= e^{-\Lambda}dX -\lambda Y dt = e^{-\Lambda}\mathcal{F}_1dt +
e^{-\Lambda}\langle \mathcal{G},d \mathcal{W}\rangle,
\]
and
\[
\begin{split}
Y(t)&= Y(0) + \int_0^t e^{-\Lambda(s)}\mathcal{F}_1(s) ds
+  \int_0^t e^{-\Lambda(s)}\langle \mathcal{G}(s),d \mathcal{W}(s)\rangle\\
&\leq  X(0) + \int_0^t e^{-\Lambda(s)}\mathcal{F}_2(s) ds
+  \int_0^t e^{-\Lambda(s)}\langle \mathcal{G}(s),d \mathcal{W}(s)\rangle.\\
\end{split}
\]
Multiplying the \db{inequality} with $e^{\Lambda(t)}$, and using the
definition of $Y$, we derive the stated stochastic version of
Gronwall's inequality.
\end{proof}
%
%
%
%
\subsection{ Uniform Bound on $\nabla \psi$ }
%
%
%
%
Using the spectral estimate of Proposition \ref{abc:spec}, we get from (\ref{e:apKEY})
\begin{equation}
\label{e:ap2}
d \|\nabla \psi \|^2
 \leq
 \Big[ C \|\nabla \psi \|^2
+  c\eps^{-1} \|R\|_{p,\mathcal{D}}^p +  2\| R \|_{p,\cD}
\|r_\eps^A\|_{q,\cD} +   C \eps^{2\sigma} \Big] dt + 2
\eps^\sigma \langle \psi , d\mathcal{W}\rangle.
\end{equation}
\db{ We apply Lemma \ref{lem:stochGW}.  Since $R(0)=0$ implies
$\nabla \psi(0)=0$, this yields}
\begin{equation*}
\begin{split}
\|\nabla \psi(t) \|^2
 &\leq
\int_0^t e^{C(t-s)} \Big[
c\eps^{-1} \|R\|_{p,\mathcal{D}}^p
+  \| R \|_{p,\cD} \|r_\eps^A\|_{q,\cD}
+   C_\text{tr} \eps^{2\sigma}
\Big] ds
+  \int_0^t e^{C(t-s)} \eps^\sigma \langle \psi , d\mathcal{W}(s)\rangle \\
&\leq e^{CT} \int_0^t \Big[ c\eps^{-1} \|R\|_{p,\mathcal{D}}^p +
\| R \|_{p,\cD} \|r_\eps^A\|_{q,\cD} +   C_\text{tr}
\eps^{2\sigma} \Big] ds +  \eps^\sigma e^{CT} \Big|\int_0^t e^{-
Cs}  \langle \psi , d\mathcal{W}(s)\rangle \Big|.
\end{split}
\end{equation*}
Furthermore, from Lemma \ref{lem:boundSI} we obtain on a subset
with high probability
\[
 \sup_{t\in[0,T_\eps]} \Big|\int_0^t  e^{- Cs} \langle \psi(s) , d\cW(s) \rangle\Big|
\leq C\eps^{\gamma-\kappa}.
\]
Thus\da{,} we arrive at
\begin{equation}
\begin{split}
\|\nabla \psi(t) \|^2
 &\leq
C \eps^{-1} \|R\|_{p,\mathcal{D}_t}^p
+ C \| R \|_{p,\cD_t} \|r_\eps^A\|_{q,\cD_t}
+   C \eps^{2\sigma} t
+ C \eps^{\sigma+\gamma-\kappa}\\
 &\leq  C[ \eps^{p\gamma-1}
+  \eps^\gamma \|r_\eps^A\|_{q,\cD_T}
+    \eps^{2\sigma}
+  \eps^{\sigma+\gamma-\kappa} ] \\
 &\leq  C[ \eps^{p\gamma-1}
+  \eps^{\sigma+\gamma-\kappa} ],
\end{split}
\end{equation}
where we used that $\gamma<\sigma$ and that $\kappa$ is
sufficiently small, together with Theorem \ref{abc:approx}.
\db{Finally,} we verified the following Lemma:
\begin{lemma}
\label{lem:45}
 For all $p\in[2,3)$, $\sigma>1$, $\kappa>0$, and $\gamma<\sigma$ we have
 \begin{equation}
 \label{e:b11}
  \|R(t) \|^2_{L^\infty (0,T_\eps,H^{-1})}
  \leq  C[ \eps^{p\gamma-1}
+  \eps^{\sigma+\gamma-\kappa} ]\da{,}
 \end{equation}
with probability larger that $1-C_\ell \eps^\ell$ for all $\ell>0$.
\end{lemma}
%
%
%
\subsection{Mean Square Bound on $\nabla R$}
%
%
%
We return to relation (\ref{e:apKEY}) and shall use the estimate
\db{
\[
-\eps^{-1}\int_0^t\int_\mathcal{D} f^\prime(u_\eps^A)R^2 dx \leq
\eps^{-2/p}\|R\|_{p,\mathcal{D}_t}^2 \da{,}
\]
presented \dbz{in \cite{abc} on p.\ 171}. Its proof is based on H\"older inequality together with
the fact that the}
set where the value of $u_\eps^A$ is not close to either $+1$ or
$-1$, has a small measure. More precisely, the
measure is controlled  by
$$
\text{measure} \Big{\{}(x,t)\in\mathcal{D}_T:f''(u_\eps^A)<0\Big{\}}\leq
C\eps,\;\;\;\;\eps\in(0,1].
$$
Therefore, integrating (\ref{e:apKEY}) \db{and using $\nabla \psi(0)=0$ (since $R(0)=0$),} we arrive at
\[
\eps  \|\nabla R\|^2_{2,\cD_t} \leq \eps^{-2/p} \|R\|_{p,\cD_t}^2,
+  c\eps^{-1} \|R\|_{p,\cD_t}^p +  \| R \|_{p,\cD_t}
\|r_\eps^A\|_{q,\cD_t} +  \eps^\sigma \int_0^t\langle \psi ,
d\mathcal{W}\rangle +   C_\text{tr} \eps^{2\sigma} t.
\]
Revoking again Lemma \ref{lem:boundSI}, we obtain on a set of
high probability
\[
\eps  \|\nabla R\|^2_{2,\cD_{T_\eps}} \leq \eps^{-2/p}
\|R\|_{p,\cD_{T_\eps}}^2 +  c\eps^{-1} \|R\|_{p,\cD_{T_\eps}}^p
+  \| R \|_{p,\cD_{T_\eps}}  \|r_\eps^A\|_{q,\cD_T} +
\eps^{\sigma+\gamma-\kappa} +   C_\text{tr} \eps^{2\sigma} T,
\]
where we used that $T_\eps\leq T$. Moreover, the
definition of $T_\eps$ implies for all $t\in[0,T_\eps]$
\[
\eps  \|\nabla R\|^2_{2,\cD_{T_\eps}} \leq \eps^{-2/p}
\eps^{2\gamma} +  c\eps^{-1} \eps^{p\gamma} +  \eps^\gamma
\|r_\eps^A\|_{q,\cD_T} +   \eps^{\sigma+\gamma-\kappa} +
C\eps^{2\sigma}.
\]
Here, the constant  depends on the final time $T$.
Using again  $\gamma<\sigma$ and $\kappa$ sufficiently small,
together with Theorem \ref{abc:approx}, we obtain
\begin{equation}
\eps  \|\nabla R\|^2_{2,\cD_{T_\eps}} \leq C[\eps^{-2/p}
\eps^{2\gamma} +  \eps^{-1} \eps^{p\gamma} +
\eps^{\sigma+\gamma-\kappa}].
\end{equation}
Note that as $p>2$, a short calculation shows that
\[
\eps^{-2/p} \eps^{2\gamma} > \eps^{-1} \eps^{p\gamma}
\quad\iff\quad \frac1p <\gamma,
\]
which we assume from now on, as we expect both $\gamma$ and $\sigma$ to be bigger that $1$.
We verified the following Lemma:
\begin{lemma}
\label{lem:46}
 For all $p\in[2,3)$, $\kappa>0$, $\sigma>1$,  and \dbz{$\gamma \in (\frac1p,\sigma)$} we have
 \begin{equation}
\label{e:b22}
  \|R(t) \|^2_{L^2 (0,T_\eps,H^{1})}
  \leq  C[ \eps^{-1-2/p + 2\gamma} +
\eps^{-1+\sigma+\gamma-\kappa} ]
 \end{equation}
with probability larger that $1-C_\ell \eps^\ell$ for all $\ell>0$.
\end{lemma}
%
%
\subsection{Final step}
%
%
In the final part of the proof it remains to show that
$T_\eps =T$ on our set of high probability. Thus, we shall use our
estimates of the previous two Lemmas to show that  $\|R\|_{p,\mathcal{D}}$ is not larger that
$\eps^\gamma$.

Observe first, that the
following trivial interpolation inequality holds true
\begin{equation}\label{e:interpol}
\|R\|_{2,\mathcal{D}}^2
=-\int_{\mathcal{D}}R\Delta\psi dx
=\int_{\mathcal{D}}\nabla R\nabla\psi dx
\leq\|\nabla R\|_{2,\mathcal{D}}\|\nabla\psi\|_{2,\mathcal{D}}.
\end{equation}

We use the Sobolev's embedding of $H^\alpha$ into $L^p$ with
$\alpha:= d(\frac12-\frac1p) = \frac{d(p-2)}{2p}$, and then
interpolate $H^\alpha$ between $L^2$ and $H^1$.
We need $\alpha  \in [0,1]$, which is
assured by $2 < p \leq 3 < \frac{2d}{(d-2)}.$ This gives,
\[\|R\|_{p,\cD}
\leq C \|R\|_{H^\alpha} \leq C \|R\|_{2,\cD}^{1-\alpha} \|\nabla
R\|_{2,\cD}^\alpha\;.
\]
Thus, using \eqref{e:interpol} we obtain
\[
\begin{split}
 \|R\|_{p,\cD}^p
& \leq C \|R\|_{2,\cD}^{\frac{2p-d(p-2)}{2}} \|\nabla
R\|_{2,\cD}^{\frac{d(p-2)}{2}} \\
&\leq
C \|\nabla\psi\|_{2,\cD}^{\frac{2p-d(p-2)}{4}}
\|\nabla R\|_{2,\cD}^{\frac{d(p-2)}{2}+ \frac{2p-d(p-2)}{4}}
= C  \|\nabla\psi\|_{2,\cD}^{\frac{2p-d(p-2)}{4}} \|\nabla
R\|_{2,\cD}^{\frac{d(p-2)+2p}{4}}\;.
\end{split}
 \]
 Integration yields
 \[
 \|R\|_{p,\cD_t}^p
\leq C\sup_{[0,t]}
 \|\nabla\psi\|_{2,\cD}^{\frac{2p-d(p-2)}{4}}
\cdot \|\nabla R\|_{2,\cD_t}^{\frac{d(p-2)+2p}{4}}.
 \]
Now we use \eqref{e:b11} and \eqref{e:b22} and arrive at
\begin{equation}\label{eqq2}
 \|R\|_{p,\cD}^p
\leq C\Big[ \eps^{p\gamma-1} +  \eps^{\sigma+\gamma-\kappa}
\Big]^{\frac{2p-d(p-2)}{8}} \Big[\eps^{-1-2/p + 2\gamma} +
\eps^{-1+\sigma+\gamma-\kappa}\Big]^{\frac{d(p-2)+2p}{8}} \;.
\end{equation}
Thus, pulling out $\eps^{2\gamma}$ from both brackets, and  as $\gamma-\sigma \geq 0$ and $\kappa$ \db{is} small, we arrive at
\[
\begin{split}
\eps^{-\gamma  } \|R\|_{p,\cD}
& \leq C\cdot
\Big[ \eps^{(p-2)\gamma-1} + \eps^{\sigma-\gamma-\kappa} \Big]^{\frac{2p-d(p-2)}{8p}}
\Big[\eps^{-1-2/p} + \eps^{-1+\sigma-\gamma-\kappa}\Big]^{\frac{d(p-2)+2p}{8p}} \\
& \leq C\cdot
\Big[ \eps^{(p-2)\gamma-1} + \eps^{\sigma-\gamma-\kappa} \Big]^{\frac{2p-d(p-2)}{8p}}
\cdot \eps^{ - (1+2/p)\frac{d(p-2)+2p}{8p} } \;.
\end{split}
\]
\dbz{By Lemmas \ref{lem:45} and \ref{lem:46} the}
 previous bound holds with probability larger than $1 - C_\ell\eps^\ell$.
In order to show that $T_\eps=T$ holds with high probability,
we need to prove that the right hand side of the previous equation is smaller than $1$.

As the second factor is larger than one, we need the first one to be smaller than $1$ and to compensate the larger factor.
Hence, we need
\[
\gamma > \frac1{p-2} \Big[ 1+ \frac{2p+d(p-2)}{2p-d(p-2)}\cdot \frac{p+2}{p}  \Big] > \frac1p
\]
and provided $\kappa$ is sufficiently small
\[
\sigma > \gamma + \frac{2p+d(p-2)}{2p-d(p-2)}\cdot \frac{p+2}{p}  > \gamma \;.
\]
Hence, the proof of theorem is now complete.

\subsection*{Acknowledgments.} Dimitra Antonopoulou and Georgia Karali are
supported by the ``ARISTEIA" Action of the ``Operational Program
Education and Lifelong Learning" co-funded by the European Social
Fund (ESF) and National Resources. {Dirk Bl\"omker thanks
MOPS at University of Augsburg for support.}

\end{document}